\documentclass[12pt]{amsart}
\usepackage{amsmath,amsthm,amsfonts,amssymb}%\usepackage{paralist,nicefrac}
\usepackage[all]{xy}

\newcommand{\C}{\mathbb{C}}
\newcommand{\Z}{\mathbb{Z}}

\newcommand{\Q}{\mathbb{Q}}
\newcommand{\F}{\mathbb{F}}
\newcommand{\A}{\mathbb{A}}
\renewcommand{\P}{\mathbb{P}}

\newcommand{\cG}{\mathcal{G}}
\newcommand{\cH}{\mathcal{H}}
\newcommand{\cD}{\mathcal{D}}

\newcommand{\Gal}{\operatorname{Gal}}

\newcommand{\Res}{\operatorname{Res}}

\newcommand{\rk}{\operatorname{rk}}
\newcommand{\ad}{\operatorname{ad}}
\newcommand{\im}{\operatorname{im}}
\renewcommand{\sc}{\operatorname{sc}}
\renewcommand{\ss}{\operatorname{ss}}
\newcommand{\der}{\operatorname{der}}
\renewcommand{\sp}{\operatorname{sp}}

\newcommand{\Aut}{\operatorname{Aut}}
\newcommand{\Int}{\operatorname{Int}}
\newcommand{\Out}{\operatorname{Out}}
\newcommand{\Dyn}{\operatorname{Dyn}}

\newcommand{\PSU}{\mathrm{PSU}}
\newcommand{\GL}{\mathrm{GL}}

\newcommand{\SL}{\mathrm{SL}}
\newcommand{\PSL}{\mathrm{PSL}}

\newcommand{\GSp}{\mathrm{GSp}}
\newcommand{\LT}{\mathfrak{g}}

\newcommand{\bD}{\mathbf{D}}
\newcommand{\bG}{\mathbf{G}}
\newcommand{\bH}{\mathbf{H}}
\newcommand{\bS}{\mathbf{S}}
\newcommand{\bZ}{\mathbf{Z}}
\newcommand{\primepower}{\ell^f}

\newtheorem{thm}{Theorem}

\newtheorem{cor}[thm]{Corollary}

\newtheorem{prop}[thm]{Proposition}

\newtheorem{remark}[thm]{Remark}
\newtheorem{lem}[thm]{Lemma}

\newtheorem*{claim*}{Claim}
\begin{document}

\title{Type A Images of Galois Representations and Maximality}

\author{Chun Yin Hui}
\address{Chun Yin Hui,
University of Luxembourg
Mathematics Research Unit
6, rue Richard Coudenhove-Kalergi
L-1359 Luxembourg
}

\author{Michael Larsen}
\address{Michael Larsen,
Department of Mathematics,
Indiana University,
Bloomington, IN
47405,
U.S.A.}

\thanks{Michael Larsen was partially supported by NSF grant DMS-1101424.}

\maketitle

\section{Introduction}

Throughout this paper, $K$ will be a number field, $\bar K$ an algebraic closure of $K$, $\mathrm{Gal}_K := \Gal(\bar K/K)$ the absolute Galois group of $K$,
$X$ a complete non-singular variety over $K$, and $i$ a non-negative integer.
Setting $W_\ell := H^i(\bar X,\Q_\ell)$, we obtain a (Serre) compatible system of Galois representations 
$\mathrm{Gal}_K\to \GL(W_\ell)$.
More generally, we consider any subsystem of $\{W_\ell\}$, i.e., any compatible system of Galois representations $\rho_\ell\colon \mathrm{Gal}_K\to \GL(V_\ell)$, 
where each $V_\ell$ is an $n$-dimensional $\mathrm{Gal}_K$-stable subspace of $W_\ell$.
Let $\Gamma_\ell := \rho_\ell(\mathrm{Gal}_K)$ denote the image,
which is a compact $\ell$-adic Lie group.   Let $\bG _\ell$ denote the Zariski closure of $\Gamma_\ell$ in $\GL_n$.

Serre conjectured and later proved \cite{ALR,Se} that 
$\Gamma_\ell = \GL_2(\Z_\ell)$ for all $\ell$ sufficiently large.
There has  been a substantial effort to generalize this
result, to two dimensional Galois representations (see, e.g., \cite{Ribet1,Ribet2}), to higher dimensional abelian varieties
(see, e.g., \cite{Serre86}),  to higher dimensional compatible systems of Galois representations (see, e.g., \cite{Hall}), and even 
to Drinfeld modules (see, e.g., \cite{PR,DP}).  In the general setting of this paper, the naive conjecture,
that $\Gamma_\ell$ is a maximal compact subgroup of $\bG _\ell(\Q_\ell)$ for all $\ell$ sufficiently large, is clearly too optimistic.
For instance, if $X = \C\P^2$, $i=4$, then $\Gamma_\ell$ consists only of the perfect squares in $\Z_\ell^\times\subset \GL_1(\Q_\ell)$.

In order to state a plausible conjecture, we introduce some additional notation.
For any algebraic group $\bG $, we denote by $\bG ^\circ$ the identity component of $\bG $, by $R(\bG )$ the radical of $\bG $, and
by $\bG ^{\ss}$ the (connected) semisimple quotient $\bG ^\circ/R(\bG ^\circ)$.
 If $\Gamma$ is a compact and Zariski-dense subgroup of $\bG (\Q_\ell)$, then 
$$\Gamma^\circ := \Gamma\cap \bG ^\circ(\Q_\ell)$$
is Zariski-dense in $\bG ^\circ$, and 
$$\Gamma^{\ss} := \Gamma^\circ/\Gamma^\circ\cap R(\bG ^\circ)(\Q_\ell)$$
is Zariski-dense
in $\bG ^{\ss}$.  By a theorem of Chevalley \cite{Pink}, $\Gamma^{\ss}$ is an open subgroup of $\bG ^{\ss}(\Q_\ell)$.
It is only the semisimple part of $\Gamma_\ell$, namely $\Gamma_\ell^{\ss}$, that we are interested in, 
since, as we have seen, we cannot expect maximality for the toral part of a Galois representation, and (conjecturally)
$\bG _\ell$ should always be reductive.
A more delicate point is that we cannot always expect  $\Gamma_\ell^{\ss}$ to be a maximal compact subgroup of $\bG ^{\ss}(\Q_\ell)$, since $\rho_\ell$
may factor through a central isogeny $\tilde{\bG}_\ell \to \bG _\ell$.  It can easily happen that the image of a maximal compact of 
$\tilde{\bG}(\Q_\ell)$ fails to be maximal in $\bG _\ell(\Q_\ell)$.

A solution to this difficulty is to define $\bG ^{\sc}$ to be the universal covering group of $\bG ^{\ss}$ and
$\Gamma^{\sc}$ to be the inverse image of $\Gamma^{\ss}$ under $\bG ^{\sc}(\Q_\ell)\to \bG ^{\ss}(\Q_\ell)$.
There is no purely group-theoretic reason to expect $\Gamma_\ell^{\sc}$ to be any less than a maximal compact subgroup of
$\bG _\ell^{\sc}(\Q_\ell)$.  The second-named author proved \cite{Larsen} that for a density $1$ set of primes $\ell$,
$\Gamma_\ell^{\sc}$ is not only a maximal compact subgroup but a \emph{hyperspecial} maximal compact.
Conjecturally, the same thing should be true for all $\ell$ sufficiently large.  In this paper, we prove it under a
group theoretic hypothesis on $\bG _\ell$.

We say that an algebraic group $\bG $ over an algebraically closed field $F$ is \emph{of type A} if every simple factor of $\bG ^{\sc}$ (defined above) is a product of groups of the form $\SL_{m_i}$.   An algebraic group $\bG $ over a general field $F$ is of type A if the group
obtained from $\bG $ by extending scalars to $\bar F$ is so.

Our main theorem is the following:

\begin{thm}
\label{main}
For all $\ell$ sufficiently large, if $\bG _\ell$ is of type A, then $\Gamma^{\sc}_\ell$ is
a hyperspecial maximal compact subgroup of $\bG _\ell^{\sc}(\Q_\ell)$ and $\bG _\ell^{\sc}$ is unramified over $\Q_\ell$.
\end{thm} 

The key idea in the proof of Theorem~\ref{main} is to compare the ``rank'' of $\Gamma_\ell$ (or its (mod $\ell$) reduction) to that of its Zariski-closure.  To some extent, this parallels the strategy of \cite{Larsen}.
A key difference is that in \cite{Larsen}, one compares two different notions of dimension rather than rank, though it should be added that
a version of dimension (somewhat different from that in \cite{Larsen}) does play some role in this paper.

We would like to thank Frank Calegari for his comments on an earlier draft of this paper.
\bigskip

\section{The Main Theorem}

There are several different definitions in the literature of the rank of an algebraic group $\bG $.  
For us, the rank of $\bG $, denoted $\rk \bG $, will always mean the absolute rank of the semisimple group $\bG ^{\ss}$.
What is special about groups of type A for our purposes is the following.

\begin{lem}
\label{A-lemma}
If $\bG $ is a connected, semisimple group of type A and $\bH$ is any proper closed subgroup of $\bG $,
then $\rk \bH < \rk \bG $.
\end{lem}

\begin{proof}
Replacing $\bG $ by its universal
covering group $\bG ^{\sc}$ and $\bH$ by $\bH\times_\bG  \bG ^{\sc}$ does not change ranks, so we may assume $\bG $
is simply connected.  Extending scalars, we may assume that $\bG  = \prod_i \SL_{m_i}$
for some sequence of integers $m_i\ge 2$.

We use induction on the number of factors.  We first treat the base case, $\bG  = \SL_m$.  
Without loss of generality, we may assume that $\bH$ is connected.  If
the unipotent radical $\mathbf{U}$ of $\bH$ is non-trivial, then by a classical theorem of Borel and Tits \cite{BT},
there exists a proper parabolic subgroup $\mathbf{P}$ of $\SL_m$ such that $\bH$ is contained in $\mathbf{P}$
and $\mathbf{U}$ is contained in the unipotent radical $\mathbf{N}$ of $\mathbf{P}$.  Thus $\bH/\mathbf{U}$ is 
a connected reductive group contained in the Levi factor
$\mathbf{M} \cong \mathbf{P}/\mathbf{N}$.  The semisimple rank of $\bH$ is the same as that of $\bH/\mathbf{U}$, which is the same as the
rank of the derived group of $\bH/\mathbf{U}$, which is isomorphic to a subgroup of the derived group of $\mathbf{M}$.
Since the derived group of $\mathbf{M}$ is a semisimple group of rank less than $\rk \bG $, any subgroup of $\mathbf{M}$
also has rank less than $\rk \bG $.

Next, we consider the case that $\bH$ is connected and reductive.  The derived group of $\bH$ is 
therefore a semisimple subgroup of $\SL_m$ whose rank equals $\rk \bH$.  However, it is well-known that 
$\SL_m$ has no proper equal-rank semisimple subgroup (because the extended Dynkin diagram of
$\SL_m$ has a transitive automorphism group).  Therefore, again $\rk \bH < \rk \bG $.

For the induction step, let
$\bH_1$ denote the image of $\bH$ under the projection to $\SL_{m_1}$ 
and $\bH^1$ denote the intersection of $\bH$ with the product of
factors other than $\SL_{m_1}$.  Thus, $\bH_1 \cong \bH/\bH^1$, and $\rk \bH = \rk \bH_1+ \rk \bH^1$.
The induction hypothesis implies that this sum equals $\rk \bG $ if and only if $\bH_1 = \SL_{m_1}$ 
and $\bH^1$ equals the product of all the remaining factors, in which case $\bH=\bG $.
\end{proof}

Let $\ell \ge 5$ be prime and $\LT$ a \emph{Lie type} (e.g., $A_n,B_n,C_n,D_n,...$). 
We define the \emph{$\ell$-dimension} function, $\dim_\ell$, the \emph{$\LT$-type $\ell$-rank} function, $\rk^{\LT}_\ell$, and 
the \emph{total $\ell$-rank} function, $\rk_\ell$, on
finite groups. The definitions of $\rk^{\LT}_\ell$ and $\rk_\ell$
can also be found from \cite[Definition 14, Remark 3.3.2]{Hui2}.
The dimension of an algebraic group $\bG /F$ as an $F$-variety is denoted $\dim\bG $, and we write $\dim\LT$ for $\dim\bG$ for any simple algebraic group $\bG$ of type $\LT$.
Let $\bar\Gamma$ be a finite simple group of Lie type in characteristic $\ell$. 
Then there exists some adjoint simple group $\bar{\bG}/\F_{\ell^f}$ so that
$$\bar\Gamma=\bar{\bG}(\F_{\ell^f})^{\der},$$
the derived group of the group of $\F_{\ell^f}$-rational points of $\bar{\bG}$.
By base change to $\bar\F_\ell$, we obtain
$$\bar{\bG}\times_{\F_{\ell^f}}\bar\F_\ell=\prod^m\bar{\bH},$$
where $\bar{\bH}$
is an $\bar\F_\ell$-adjoint simple group of some Lie type $\mathfrak{h}$.
We then set the $\ell$-dimension of $\bar\Gamma$ to be 
$$\dim_\ell\bar\Gamma:=f\cdot\dim\bar{\bG},$$
the $\LT$-type $\ell$-rank of $\bar\Gamma$ to be
\begin{equation*}
\rk^{\LT}_\ell\bar\Gamma :=\left\{ \begin{array}{lll}
 f\cdot\rk \bar{\bG} &\mbox{if}\hspace{.1in} \LT=\mathfrak{h},\\
 0 &\mbox{otherwise,}
\end{array}\right.
\end{equation*}
and the total $\ell$-rank of $\bar\Gamma$ to be 
$$\rk_\ell\bar\Gamma:=\sum_{\LT}\rk_\ell^{\LT}\bar\Gamma.$$
 %(i.e. $\sigma$ permutes the simple components of $\bG $ over $\bar\F_\ell$ in a single orbit).
For example, $\PSL_n(\F_{\ell^f})$ and $\PSU_n(\F_{\ell^f})$ both have $\ell$-dimension $f(n^2-1)$, $A_{n-1}$-type $\ell$-rank and total $\ell$-rank $f(n-1)$.
For simple groups which are not of Lie type in characteristic $\ell$, we define the $\ell$-dimension and $\LT$-type $\ell$-rank to be zero.  
In particular, the total $\ell$-rank of $\Z/p\Z$ is zero, even if $p=\ell$.
We extend the definitions
to arbitrary finite groups by defining the $\ell$-dimension, $\LT$-type $\ell$-rank, and total $\ell$-rank of any finite group to be the sum of the ranks of its composition factors.
This definition makes it clear that $\dim_\ell$, $\rk_\ell^{\LT}$, and $\rk_\ell$ are additive on short exact sequences of groups.  In particular, the $\ell$-dimension and the total $\ell$-rank of
every solvable finite group are zero.

We can further extend the definitions to certain infinite profinite groups, including compact subgroups of $\GL_n(\Q_\ell)$, as follows.
If $\Gamma$ is a finitely generated profinite group which contains an open pro-solvable subgroup, then 
\begin{equation}
\label{defs}
\begin{split}
\dim_\ell\Gamma &:= \dim_\ell\Gamma/\Delta, \\
\rk_\ell^{\LT} \Gamma &:= \rk_\ell^{\LT} \Gamma/\Delta, \\
\rk_\ell \Gamma &:= \rk_\ell \Gamma/\Delta
\end{split}
\end{equation}
for any normal, pro-solvable, open subgroup $\Delta$ of $\Gamma$. Since for any profinite group,
open normal subgroups are cofinal among all open subgroups and if $\Delta_1\subset \Delta_2$ are normal, prosolvable, and open,
then  $\Delta_2/\Delta_1$ is finite and solvable, the right hand side in each formula in (\ref{defs}) is well-defined.  
As the $\ell$-dimension and the total $\ell$-rank of every pro-$\ell$ group is zero, we have
\begin{equation*}
\begin{split}
\dim_\ell\Gamma_\ell &= \dim_\ell\bar\Gamma_\ell,\\
\rk_\ell^{\LT} \Gamma_\ell &= \rk_\ell^{\LT} \bar\Gamma_\ell,\\
\rk_\ell \Gamma_\ell &= \rk_\ell \bar\Gamma_\ell,
\end{split}
\end{equation*}
where $\bar\Gamma_\ell$
denotes the image in $\GL_n(\F_\ell)$ the reduction (mod $\ell$) of $\Gamma_\ell$ with
respect to any $\Z_\ell$-lattice in $\Q_\ell^n$ stabilized by $\Gamma_\ell$.  

If $\Gamma$ is a compact subgroup of $\GL_n(\Q_\ell)$ and $\Delta$ is a closed normal subgroup, we fix
a normal open, pro-$\ell$ subgroup $\Gamma_0 $ of $\Gamma$ and define 
\begin{equation*}
\begin{aligned}
\Delta_0 & := \Delta\cap \Gamma_0 \\
(\Gamma/\Delta)_0 & := \Gamma_0/\Delta_0\subset \Gamma/\Delta
\end{aligned}
\end{equation*}
The two groups above are pro-$\ell$ since closed subgroups and quotient groups of a pro-$\ell$ group are still pro-$\ell$ \cite[Proposition 1.11]{DDMS}. We have a short exact sequence of finite groups
$$0\to \Delta/\Delta_0\to \Gamma/\Gamma_0\to (\Gamma/\Delta)/(\Gamma/\Delta)_0\to 0,$$
and therefore 

\begin{equation*}
\begin{split}
\dim_\ell \Gamma   &= \dim_\ell \Delta + \dim_\ell \Gamma/\Delta,\\
\rk_\ell^{\LT} \Gamma  &= \rk_\ell^{\LT} \Delta + \rk_\ell^{\LT} \Gamma/\Delta,\\
\rk_\ell\Gamma&=\rk_\ell\Delta+\rk_\ell\Gamma/\Delta.
\end{split}
\end{equation*}

Our basic results on $\dim_\ell$, $\rk_\ell^{\LT}$, and $\rk_\ell$ of finite groups are the following.

\begin{prop} \label{points}
Let $\ell\geq5$ be a prime and $\bar{\bG}$ a connected algebraic group over $\mathbb{F}_{\primepower}$. 
The composition factors of $\bar{\bG}(\mathbb{F}_{\primepower})$ are cyclic groups and finite simple groups of Lie type of characteristic $\ell$. Moreover, let $m_\LT$ be the number of almost simple factors of $\bar{\bG}^{\sc}\times_{\F_{\primepower}}\bar{\F}_{\primepower}$ of simple type $\LT$. Then the following equations hold:
\begin{enumerate}
\item[(i)] $\rk_\ell^\LT\bar{\bG}(\F_{\primepower})=m_\LT f\cdot\rk\LT$.
\item[(ii)] $\rk_\ell\bar{\bG}(\F_{\primepower})=f\cdot\rk\bar{\bG}$.
\item[(iii)] $\dim_\ell \bar{\bG}(\F_{\primepower})=f\sum_{\LT}m_{\LT}\cdot \dim\LT= f\cdot\dim \bar{\bG}^{\ss}$.
\end{enumerate}
\end{prop}

\begin{proof}
Since the $\ell$-dimension and the total $\ell$-rank of solvable groups are zero, we perform the following reductions. Let $R^\mathrm{u}(\bar{\bG})$ be the unipotent radical of $\bar{\bG}$. We have exact sequence
$$1\to R^\mathrm{u}(\bar{\bG})(\F_{\primepower})\to \bar{\bG}(\F_{\primepower})\to \bar{\bG}/R^\mathrm{u}(\bar{\bG})(\F_{\primepower})\to H^1(\F_{\primepower}, R^\mathrm{u}(\bar{\bG})).$$
Since $R^\mathrm{u}(\bar{\bG})(\F_{\primepower})$ is solvable and $H^1(\F_{\primepower}, R^\mathrm{u}(\bar{\bG}))=0$, we may assume $\bar{\bG}$ reductive. Let $\bar{\mathbf{Z}}$ be the center of $\bar{\bG}$. We have exact sequence
$$1\to \bar{\mathbf{Z}}(\F_{\primepower})\to \bar{\bG}(\F_{\primepower})\to \bar{\bG}/\bar{\mathbf{Z}}(\F_{\primepower})\to H^1(\F_{\primepower}, \bar{\mathbf{Z}}).$$
Since $\bar{\mathbf{Z}}(\F_{\primepower})$ and $H^1(\F_{\primepower}, \bar{\mathbf{Z}})$ are abelian, we may assume $\bar{\bG}=\bar{\bG}^{\ss}$, i.e., $\bar{\bG}$ is semisimple.
If $\bar{\bD}$ denotes the kernel of $\bar{\bG}^{\sc}\to \bar{\bG}^{\ss}$, then we have an exact sequence
$$1\to \bar{\bD}(\F_{\primepower})\to \bar{\bG}^{\sc}(\F_{\primepower})\to \bar{\bG}^{\ss}(\F_{\primepower})\to H^1(\F_{\primepower}, \bar{\bD}),$$
so we may assume $\bar{\bG} = \bar{\bG}^{\sc}$.
Then the assertion on composition factors and parts (i) and (ii) follow from \cite[Proposition 3.3.3]{Hui2}. Part (iii) follows easily from the definition.
\end{proof}

\begin{thm}
\label{finite}
Let $\bar{\bG}$ be a connected algebraic group over $\F_{\ell}$
and $\bar\Gamma\subset \bar{\bG}(\F_\ell)$ a subgroup. If $\ell$ is sufficiently large compared to $\dim\bar{\bG}$, then
the following hold:
\begin{enumerate}
\item[(i)] $\rk_\ell \bar \Gamma\le  \rk \bar{\bG}$.
\item[(ii)] If $\rk_\ell \bar \Gamma = \rk \bar{\bG}$ and $\bar{\bG}$ is simply connected and semisimple of type A, then 
$\bar\Gamma=\bar{\bG}(\F_{\ell})$.
\item[(iii)] If $\rk_\ell^{\LT}\bar\Gamma =\rk_\ell^{\LT} \bar{\bG}(\F_{\ell})$ for all Lie types $\LT$ and $\bar{\bG}$ is simply connected and semisimple, then $\bar\Gamma=\bar{\bG}(\F_{\ell})$.
\item[(iv)] If $\dim_\ell \bar\Gamma\geq \dim\bar{\bG}$, then $\bar{\bG}$ is semisimple and equality holds.
\end{enumerate}
\end{thm}

\begin{proof}
We recall some ideas of Nori \cite{Nori}.  Let $k$ be a positive integer.
If $\ell > k$ and $x\in M_k(\F_\ell)$ is nilpotent, the formula  
$$\phi_x(t) := \exp tx = 1 + tx + \frac{t^2}2 x^2 + \cdots + \frac{t^{k-1}}{(k-1)!} x^{k-1}$$ 
defines an $\F_\ell$-morphism
of algebraic groups $\phi_x\colon \A^1\to \GL_k$.  In particular, if $x\neq 0$, then $\phi_x(1)$ is of
order $\ell$.   Conversely, every element $u$ of order $\ell$ in $\GL_k(\F_\ell)$
is of the form $\exp(x)$ for some nilpotent $x\in M_k(\F_\ell)$, namely,
$$x = \log u := (u-1) - \frac{(u-1)^2}2 + \cdots + (-1)^k \frac{(u-1)^{k-1}}{k-1}.$$
A closed subgroup of $\GL_k$ over $\F_\ell$ is \emph{exponentially generated} if it is generated
by subgroups of the form $\phi_x(\A^1)$.   Since every quotient group of an exponentially generated group
is exponentially generated, every exponentially generated group is connected, every reductive exponentially generated group is semisimple, and more generally, every exponentially generated group is the extension of a semisimple group by a unipotent group.

For any subgroup $\bar\Delta\subset \GL_k(\F_\ell)$, we define $\bar\Delta^+$ to be the subgroup
of $\bar\Delta$ generated by elements of order $\ell$.  Clearly $\bar\Delta^+$ is a normal subgroup,
and by \cite[Theorem~C]{Nori}, $\bar\Delta/\bar\Delta^+$ contains an abelian normal subgroup of
index bounded by a constant depending only on $k$.  In particular, if $\ell$ is sufficiently large compared to $k$, as we assume henceforth, then
\begin{equation*}
\begin{split}
\rk_\ell^{\LT}\bar\Delta &= \rk_\ell^{\LT}\bar\Delta^+,\\  
\rk_\ell \bar\Delta &= \rk_\ell \bar\Delta^+,\\  
\dim_\ell \bar\Delta &= \dim_\ell \bar\Delta^+.
\end{split}
\end{equation*}
If $\bar{\bH}$ denotes the algebraic subgroup of $\GL_k$ generated
by $\phi_x(\A^1)$ as $\exp(x)$ ranges over all elements of order $\ell$ in $\bar\Delta$
we have $\bar{\bH}(\F_\ell)^+ = \bar\Delta^+$ \cite[Theorem~B]{Nori}.

If $\bar{\bS}/\F_\ell$ is a simply connected semisimple subgroup of $\GL_k$,
then $\bar{\bS}(\F_\ell)\subset \GL_k(\F_\ell)$ is generated by its unipotent elements \cite[Theorem~12.4]{Steinberg}, each of the form $\exp(x)$
for some nilpotent matrix $x$.  Let $\bar{\bS}^+$ denote the algebraic subgroup of $\GL_k$ generated by the corresponding $\phi_x(\A^1)$.
Assuming, as always, that $\ell$ is  sufficiently large, we have 
$$\bar{\bS}^+(\F_\ell)^+ = \bar{\bS}(\F_\ell)^+ = \bar{\bS}(\F_\ell)$$
 \cite[Theorem~B]{Nori}.  As the index $[\bar{\bS}^+(\F_\ell):\bar{\bS}^+(\F_\ell)^+]$ is uniformly bounded \cite[3.6(v)]{Nori}
and $\bar{\bS}^+$ and $\bar{\bS}$ are connected,
the estimates on the orders of $\bar{\bS}^+(\F_\ell)$ and $\bar{\bS}(\F_\ell)$ \cite[3.5]{Nori}  imply 
first that
$\dim \bar{\bS}^+ = \dim \bar{\bS}$ and then that
$\bar{\bS}^+(\F_\ell) = \bar{\bS}^+(\F_\ell) ^+ = \bar{\bS}(\F_\ell)$.  Now, $\bar{\bS}^+$ is exponentially generated and therefore
is the extension of  a semisimple group by a connected unipotent group $\bar{\mathbf{U}}$.  As $\bar{\bS}^+(\F_\ell)\cong \bar{\bS}(\F_\ell)$
has no solvable normal subgroups except subgroups of the center of $\bar{\bG}$ and as the order of this center is $\le k < \ell$, it follows that $\bar{\mathbf{U}}$ is trivial, so $\bar{\bS}^+$ is semisimple.
Because their groups of $\F_\ell$-points are the same, $\bar{\bS}$ and $\bar{\bS}^+$
belong to the same isogeny class.

Now we apply this theory to a subgroup $\bar \Gamma$ 
of $\bar{\bG}(\F_\ell)$ where $\bar{\bG}$ is a connected group over $\F_\ell$, and $\ell$ is still assumed sufficiently large in terms of $k$.
As the image  $\bar\Gamma^{\ss}$ of $\bar\Gamma$ in $\bar{\bG}^{\ss}(\F_\ell)$ is the quotient of $\bar\Gamma$
by a solvable group, we have 
\begin{equation*}
\begin{split}
\rk_\ell^{\LT} \bar\Gamma^{\ss} &= \rk_\ell^{\LT} \bar\Gamma, \\
\rk_\ell \bar\Gamma^{\ss} &= \rk_\ell \bar\Gamma, \\
\dim_\ell \bar\Gamma^{\ss} &= \dim_\ell \bar\Gamma,
\end{split}
\end{equation*}
so replacing
$\bar{\bG}$ and $\bar \Gamma$ by $\bar{\bG}^{\ss}$ and $\bar\Gamma^{\ss}$ respectively, we may assume without loss
of generality that $\bar{\bG}$ is semisimple.
Let $\bar \Gamma^{\sc}$ denote the pullback of $\bar\Gamma$ to $\bar{\bG}^{\sc}(\F_\ell)$.  As $\bar\Gamma$ contains 
as a subgroup of bounded index the quotient of $\bar \Gamma^{\sc}$ by a subgroup of bounded order, we have
\begin{equation*}
\begin{split}
\rk_\ell^{\LT} \bar\Gamma^{\sc} &= \rk_\ell^{\LT} \bar\Gamma, \\
\rk_\ell \bar\Gamma^{\sc} &= \rk_\ell \bar\Gamma, \\
\dim_\ell \bar\Gamma^{\sc} &= \dim_\ell \bar\Gamma.
\end{split}
\end{equation*}

By the classification of semisimple groups
and their representations, there exists a faithful representation $\bar{\bG}^{\sc}\hookrightarrow \GL_k$ defined over $\F_\ell$ for
some $k$ bounded in terms of $\rk \bar{\bG}$.
Let $\bar{\bH}$ and $\bar{\bG}^+$ denote respectively the $\F_\ell$-subgroup of $\GL_k$ generated by $\phi_x(\A^1)$
as $\exp x$ ranges over all order-$\ell$ elements of $\bar \Gamma^{\sc}$ and 
of $\bar{\bG}(\F_\ell)$.  For part (iv), we have by hypothesis and Proposition~\ref{points} (iii),
\begin{align*}
\dim_\ell(\bar\Gamma^{\sc})^+&=\dim_\ell\bar\Gamma^{\sc} = \dim_\ell \bar\Gamma \ge \dim \bar{\bG} \ge \dim \bar{\bG}^{\sc} \ge \dim \bar{\bH} \\
&\ge \dim \bar{\bH}^{\ss} =   \dim_\ell \bar{\bH}(\F_\ell) =  \dim_\ell \bar{\bH}(\F_\ell)^+ = \dim_\ell(\bar\Gamma^{\sc})^+.
\end{align*}
Thus, $\dim \bar{\bG} = \dim \bar{\bG}^{\sc}$, which implies that $\bar{\bG}$ is semisimple.

For parts (i)--(iii), replacing $\bar\Gamma$ and $\bar{\bG}$ by $\bar\Gamma^{\sc}$ and $\bar{\bG}^{\sc}$ respectively,
we may assume that $\bar{\bG}$ is a simply connected semisimple $\F_\ell$-subgroup of $\GL_k$, where $k$ is bounded in terms of $\dim\bar{\bG}$.
Replacing $\bar\Gamma$ by $\bar\Gamma^+$, we may further assume that 
$\bar{\bH}(\F_\ell)^+=\bar\Gamma$. 
Since $\bar{\bH}\subset \bar{\bG}^+$ by construction and $\bar{\bG}$ and $\bar{\bG}^+$ are of the same isogeny type,
we have 
\begin{equation}
\label{chain}
\rk_\ell\bar\Gamma = \rk_\ell \bar{\bH}(\F_\ell)= \rk \bar{\bH}\leq \rk\bar{\bG}^+=\rk \bar{\bG}
\end{equation}
This proves (i).
For (ii), $\rk_\ell \bar \Gamma = \rk \bar{\bG}$ and (\ref{chain}) imply
$$\rk \bar{\bH} = \rk\bar{\bG} = \rk\bar{\bG}^+.$$
We further assume
that $\bar{\bG}$ (and therefore $\bar{\bG}^+$) is of type A.   As $\bar{\bH}\subset \bar{\bG}^+$ have the same rank, by Lemma~\ref{A-lemma},
$\bar{\bH}= \bar{\bG}^+$.
Since $\bar{\bH}(\F_\ell) = \bar{\bG}^+(\F_\ell)= \bar{\bG}(\F_\ell)$ is generated by order $\ell$ elements, we conclude that 
$$
\bar \Gamma = \bar{\bH}(\F_\ell) = \bar{\bG}^+(\F_\ell) = \bar{\bG}(\F_\ell).$$

For any $\F_\ell$-semisimple group $\bar{\mathbf{S}}$, we denote by $\rk^{\LT}\bar{\mathbf{S}}$ the rank of the subgroup of $\bar{\mathbf{S}}\times_{\F_\ell}\bar\F_\ell$ generated by almost simple normal subgroups of type $\LT$. 
By Proposition \ref{points}(i), $\rk^{\LT}\bar{\mathbf{S}} = \rk^{\LT}_\ell\bar{\mathbf{S}}(\F_\ell)$.
By the hypothesis of part (iii), we obtain for all Lie types $\LT$  that 
\begin{align*}
\rk^{\LT} \bar{\bH}&=\rk_\ell^{\LT} \bar{\bH}(\F_\ell)=\rk_\ell^{\LT} \bar{\bH}(\F_\ell)^+=\rk_\ell^{\LT} \bar\Gamma^+ \\
&=\rk_\ell^{\LT} \bar\Gamma
=\rk^{\LT}_\ell \bar{\bG}(\F_\ell)=\rk^{\LT} \bar{\bG}=\rk^{\LT} \bar{\bG}^+.
\end{align*}
Since $\bar{\bG}^+$ is semisimple and $\bar{\bH}\subset\bar{\bG}^+$, this implies $\bar{\bH}=\bar{\bG}^+$ by dimension considerations.
As $\bar{\bG}$ is simply connected, it follows that
$$\bar \Gamma = \bar{\bH}(\F_\ell) = \bar{\bG}^+(\F_\ell) = \bar{\bG}(\F_\ell).$$

%From the hypothesis of (iv) and Proposition \ref{points}(iii), we obtain 
%$$\dim_\ell\bar\Gamma = \dim_\ell \bar{\bH}(\F_\ell)= \dim \bar{\bH}\leq \dim\bar{\bG}'=\dim \bar{\bG}$$
%Therefore, $\dim\bar{\bG}= \dim\bar{\bG}^{\sc}$ and $\dim\bar{\bG}$ is semisimple.

\end{proof}

\begin{cor}
\label{finite-cor}
Let $\bar{\bG}$ be a connected algebraic group over $\F_{\primepower}$
and $\bar\Gamma\subset \bar{\bG}(\F_{\primepower})$ a subgroup. If $\ell$ is sufficiently large compared to $\dim\bar{\bG}$ and $f$, then
the following hold:
\begin{enumerate}
\item[(i)] $\rk_\ell \bar \Gamma\le  f\rk \bar{\bG}$.
\item[(ii)] If $\rk_\ell \bar \Gamma = f\rk \bar{\bG}$ and $\bar{\bG}$ is simply connected and semisimple of type A, then 
$\bar\Gamma=\bar{\bG}(\F_{\primepower})$.
\item[(iii)] If $\rk_\ell^{\LT}\bar\Gamma =\rk_\ell^{\LT} \bar{\bG}(\F_{\primepower})$ for all Lie types $\LT$ and $\bar{\bG}$ is simply connected and semisimple, then $\bar\Gamma=\bar{\bG}(\F_{\primepower})$.
\item[(iv)] If $\dim_\ell \bar\Gamma\geq f\dim\bar{\bG}$, then $\bar{\bG}$ is semisimple and equality holds.
\end{enumerate}
\end{cor}

\begin{proof}
This follows immediately from Theorem~\ref{finite} by replacing $\bar{\bG}$ with the Weil restriction of scalars $\Res_{\F_{\primepower}/\F_\ell} \bar{\bG}$.
\end{proof}

\begin{prop}
\label{split}
Let $F$ be a local field in characteristic zero. For any connected semisimple algebraic group $\bG /F$, there exists a finite extension $F'$ of $F$ such that $\bG $ splits over $F'$ and the degree $[F':F]$ is bounded by a constant that depends only on $\dim\bG$.
\end{prop}

\begin{proof}
For archimedan local fields, the proposition is trivial, so we assume $F$ is non-archimedean.
We may assume $\bG $ is simply connected. 
Then $\bG $ is an $F$-form of $\bG ^{\sp}$, which is semisimple, simply connected, and split over $F$. 
The form of $\bG $ is represented by a Galois cohomology class 
$$[c_\sigma]\in H^1(F,\mathrm{Aut}(\bG ^{\sp})).$$ 
We will find an extension $F'$ of $F$ such that the image of $[c_\sigma]$ under the restriction map
$$\mathrm{Res}:H^1(F,\mathrm{Aut}(\bG ^{\sp}))\rightarrow H^1(F',\mathrm{Aut}(\bG ^{\sp}))$$
is the trivial class. 

Let $\Dyn_{\bG}$ and $\bG ^{\ad}$ be respectively the Dynkin diagram  and the adjoint quotient of $\bG ^{\sp}$. 
The short exact sequence
$$1\to \Int \bG^{\sp}(\bar F)\to \Aut \bG^{\sp}(\bar F) \to \Out \bG^{\sp}(\bar F)\to 1$$
is split (\cite[Cor.~2.14]{Springer}), and we can identify the inner automorphism group $\Int \bG^{\sp}(\bar F)$ with $\bG ^{\ad}(\bar F)$
and $\Out \bG^{\sp}(\bar F)$ with $\Aut\Dyn_{\bG}$.
For any $F'$, the non-abelian cohomology sequence \cite[Chap. 1, Prop.~38]{GC} gives an exact sequence of pointed sets
$$1\rightarrow H^1(F',\bG ^{\ad})\stackrel{i}{\rightarrow} H^1(F',\mathrm{Aut}(\bG ^{\sp}))\stackrel{j}{\rightarrow} H^1(F',\mathrm{Aut}(\Dyn_{\bG}))$$

Since $\bG $ is semisimple and $\rk \bG $ is bounded in terms of $\dim \bG$, 
the size of $\mathrm{Aut}(\Dyn_{\bG})$ is bounded by a constant depending on $\dim \bG$. We can choose an extension $F'$ of $F$ of bounded degree such that the Galois action of $\mathrm{Gal}_{F'}$ on $\mathrm{Aut}(\Dyn_{\bG})$ is trivial 
and  $j(\mathrm{Res}([c_\sigma]))$ corresponds to the trivial homomorphism of $\mathrm{Hom}(\mathrm{Gal}_{F'},\mathrm{Aut}(\Dyn_{\bG}))=H^1(F',\mathrm{Aut}(\Dyn_{\bG}))$. 
Thus, $\mathrm{Res}([c_\sigma])$ is in the image of $i$ and we may identify it with an element of $H^1(F',\bG ^{\ad})$. 

Let $\mathbf{Z}$ be the center of $\bG^{\sp}$. 
As  $\bG^{\sp}$ is split, we can identify $\bZ$ with a product of factors of the form $\mu_{d_i}$.
We have a short exact sequence
$$1\rightarrow \mathbf{Z}\rightarrow \bG ^{\sp}\rightarrow \bG ^{\ad}\rightarrow 1.$$
The long exact sequence for central extensions in non-abelian cohomology (\cite[Chap.~1,~Prop.~43]{GC}) shows that
$$H^1(F',\bG ^{\sp})\rightarrow H^1(F',\bG ^{\ad})\rightarrow H^2(F',\mathbf{Z})$$
is exact. 
As $H^1(F',\bG ^{\sp})=0$ for every $\ell$-adic field $F'$ \cite[Chap. 3 $\mathsection 3$]{GC}, we have
$$H^1(F',\bG ^{\ad}) \subset H^2(F',\mathbf{Z}) =\bigoplus_i H^2(F',\mu_{d_i}) =  \bigoplus_i \mathrm{Br}(F')[d_i].$$
By \cite[Chap.~XIII,~Prop.~7]{LF}, the restriction map associated to any extension of $F'$ of degree divisible by each $d_i$ kills the 
class of any $H^1(F',\bG ^{\ad})$.

\end{proof}

Our basic result on $\rk_\ell$ for $\ell$-adic groups is the following.

\begin{thm}
\label{l-adic}
Let $\bG $ be a connected algebraic group over $\Q_\ell$ and
$\Gamma \subset \bG (\Q_\ell)$ a compact subgroup of $\bG (\Q_\ell)$. If $\ell$ is sufficiently large compared to $\rk \bG $, then the following hold:
\begin{enumerate}
\item[(i)] $\rk_\ell \Gamma\le \rk \bG $.
\item[(ii)] If $\rk_\ell \Gamma = \rk \bG $ and $\bG $ is simply connected and semisimple of type A,
then $\Gamma$ is a hyperspecial maximal compact subgroup of $\bG (\Q_\ell)$.
\end{enumerate}
\end{thm}

\begin{proof}
Replacing $\bG $ and $\Gamma$ by $\bG ^{\sc}$ and $\Gamma^{\sc}$ do not change ranks,
so we may in any case assume that $\bG $ is simply connected and semisimple. 
By \cite[3.2]{Tits}, if $F$ is a finite extension of $\Q_\ell$, 
every maximal compact subgroup of $\bG (F)$ is the stabilizer $\bG (F)^x$ of a vertex $x$  in 
the Bruhat-Tits building of $\bG /F$.   
There exists a smooth affine group scheme $\cG$ over the ring of integers $O_F$ of $F$
and an isomorphism $\iota$ from the generic fiber of $\cG$ to $\bG _{F}$ 
such that $\iota(\cG(O_F)) = \bG (F)^x$ \cite[3.4.1]{Tits}. 
Let $\F_{\primepower}$ be the residue field of $O_F$. 
As $\bG $ is simply connected, 
the special fiber $\cG_{\F_{\primepower}}$ is connected \cite[3.5.2]{Tits}. 
The maximal compact subgroup is hyperspecial if and only if $\cG_{\F_{\primepower}}$ is reductive \cite[3.8.1]{Tits}, 
in which case it has the same root datum as the generic fiber \cite[XXII, 2.8]{SGA3}. 
The completion of the smooth group scheme $\cG$ along the identity gives a formal group law
whose elements identify with the kernel of 
$$r\colon \cG(O_F)\to \cG(\F_{\primepower})$$ 
\cite[1.2]{Serre64}, so 
$\ker r$ is a pro-$\ell$ group \cite[III, \S7, Prop.~5]{Bo}.
Since $r$ is surjective by Hensel's lemma, we obtain
$$\rk_\ell\cG(O_F)= \rk_\ell \cG(\F_{\primepower}).$$

If $\Gamma$ is a compact subgroup of $\bG (\Q_\ell)$, 
it fixes some vertex $x$ in the Bruhat-Tits building of $\bG /\Q_\ell$.
By applying Proposition \ref{split} and \cite[2.4]{Larsen} to $\bG /\Q_\ell$,  
there exists a finite extension $F/\Q_\ell$ whose residue class degree $f$ is bounded by a constant $C$ depending on $\rk \bG $ such that 
$\bG $ splits over $F$ and the vertex $x$ becomes hyperspecial in the Bruhat-Tits building of $\bG /F$. 
Hence, $\Gamma$ is contained in a hyperspecial maximal subgroup of $\bG (F)$, 
which we identify with $\cG(O_F)$. 

By Bruhat-Tits theory \cite[3.4.3]{Tits}, there exists a smooth group scheme $\cH$ over $\Z_\ell$
with generic fiber $\bG$ such that for every unramified finite extension $E/\Q_\ell$,
the stabilizer in $\bG(E)$  of the image of $x$ under the map of buildings $B(\bG,\Q_\ell)\to B(\bG,E)$
is $\cH(O_E)$.

Let $E$ denote the maximal subfield of $F$ unramified over $\Q_\ell$. 
Let $\F_{\primepower}$ be the residue field of $O_E$, which is also the residue field of $O_F$. 
Define surjections
\begin{align*}
r_0&\colon \cG(O_F)\to\cG(\F_{\primepower}) \\
r_1&\colon\cH(\Z_\ell)\to\cH(\F_\ell).
\end{align*}
Since we have $\cH(O_E)\subset \cG(O_F)$ and $\Gamma\subset \cH(\Z_\ell),$
we obtain by Corollary~\ref{finite-cor}(i) and Proposition~\ref{points}(ii) and 
 that 
\begin{equation}
\label{star}
\begin{split}
 f\rk \bG &=f\rk\cG_{\F_{\primepower}}  \geq \rk_\ell r_0(\cH(O_E)) =\rk_\ell\cH(O_E)=\rk_\ell\cH(\F_{\primepower}) \\
               &= f\rk \cH_{\F_{\primepower}} =f\rk \cH_{\F_\ell}  \geq f\rk_\ell r_1(\Gamma) =f\rk_\ell \Gamma
\end{split}
\end{equation}
for all sufficiently large $\ell$. Assertion (i) follows immediately.

For (ii), since $\rk_\ell\Gamma=\rk \bG $,  equality must hold in both inequalities in (\ref{star}). 
Since 
$$r_0(\cH(O_E))\subset \cG(\F_{\primepower})=\cG_{\F_{\primepower}} (\F_{\primepower}),$$
$\cG_{\F_{\primepower}}$ is of type A, and 
$f$ is bounded by a constant depending only on $\dim G$,  by Corollary \ref{finite-cor}(ii),
$r_0(\cH(O_E))=\cG(\F_{\primepower})$, assuming, as always, that $\ell$ is sufficiently large compared to $\rk \bG $. 
Since $\cH(\F_{\primepower})$ is also a quotient of $\cH(O_E)$ by a pro-$\ell$ group, 
we conclude that the composition factors of $\cH(\F_{\primepower})$ and $\cG(\F_{\primepower})$ are identical modulo cyclic groups.

Since $\cG_{\F_{\primepower}}$ is semisimple, we have 
$$\dim_\ell \cH(\F_{\primepower})=\dim_\ell \cG(\F_{\primepower})=f\dim \cG_{\F_{\primepower}}$$
by Proposition \ref{points}(iii). 
Since $\cG$ and $\cH$ are flat, this means
$$\dim_\ell \cH(\F_{\primepower})=f\dim \cG_{\F_{\primepower}}=f \dim\bG = f\dim \cH_{\F_\ell}.$$
As the special fiber of $\cH$ is connected (\cite[3.5.3]{Tits}), we can apply Corollary~\ref{finite-cor}(iv) to $\bar\Gamma := \cH(\F_{\primepower})$
and $\bar{\bG} := \cH_{\F_\ell}$
and deduce that $\cH_{\F_{\ell}}$ is semisimple.
This implies $\cH(\Z_\ell)$ is a hyperspecial maximal compact subgroup of $\bG (\Q_\ell)$ containing $\Gamma$.
Hence, $\cH_{\F_\ell}$ is a connected, simply connected and semisimple type A $\F_\ell$-group. 
As $r_1(\Gamma)$ and $\cH_{\F_\ell}$ both have $\ell$-rank equal to $\rk \bG $,  by Corollary~\ref{finite-cor}(ii),
$$r_1(\Gamma)=\cH_{\F_\ell}(\F_\ell)=\cH(\F_\ell).$$
By the main theorem of \cite{Vasiu}, this implies
$$\Gamma = \cH(\Z_\ell).$$

\end{proof}

We can now prove Theorem~\ref{main}.

\begin{proof}
We observe first that if $\mathbf{L}$ is a connected algebraic group and $\mathbf{D}$ is a normal subgroup of $\mathbf{L}$
then $\mathbf{D}/\mathbf{D}^\circ$ is solvable.  Indeed, if $\mathbf{R}$ is the maximal normal solvable subgroup of $\mathbf{L}$,
every subquotient of $\mathbf{R}$ is again solvable, so it suffices to prove that
$$\im(\mathbf{D}\to \mathbf{L}/\mathbf{R})/\im(\mathbf{D}^\circ\to \mathbf{L}/\mathbf{R})$$
is solvable.  It therefore suffices to prove that the component group of every normal
subgroup of $\mathbf{L}/\mathbf{R}$ is solvable.  However, $\mathbf{L}/\mathbf{R}$ is adjoint semisimple, so every normal subgroup of $\mathbf{L}/\mathbf{R}$ is 
likewise adjoint semisimple and, in particular, connected.

Now, let $W_\ell := H^i(\bar X,\Q_\ell)$.  Let $\Lambda_\ell$ denote the image of $\mathrm{Gal}_K$ in $\GL(W_\ell)$
and $\mathbf{L}_\ell$ denote the Zariski-closure of $\Lambda_\ell$.  By  \cite[6.14]{LP},
the order of the component group $\mathbf{L}_\ell/\mathbf{L}^\circ_\ell$ does not depend on $\ell$.
Therefore, replacing $K$ by a finite extension, we may and do assume that $\mathbf{L}_\ell$ is connected for
all $\ell$.
For each $\ell$, we suppose chosen a $\mathrm{Gal}_K$-stable subspace $V_\ell$ such that the resulting system of
Galois representations is compatible in the sense of Serre.
Let $\Gamma_\ell$ denote the image
of $\mathrm{Gal}_K$ in $\GL(V_\ell)$.
Let $\Delta_\ell$ denote the kernel of the surjective homomorphism $\Lambda_\ell\to \Gamma_\ell$.
Let $\bG _\ell$ and $\mathbf{D}_\ell$ denote the Zariski-closures of $\Gamma_\ell$ and $\Delta_\ell$
respectively.  As $\mathbf{L}_\ell$ is connected, the same is true of its quotient $\bG _\ell$.
As $\mathbf{D}_\ell/\mathbf{D}_\ell^\circ$ is solvable, the total $\ell$-rank of $\Delta_\ell$ is the same as that
of $\Delta_\ell^\circ := \Delta_\ell\cap \mathbf{D}_\ell^\circ(\Q_\ell)$.

By \cite[Theorem A]{Hui2}, for $\ell$ is sufficiently large, $\rk_\ell\Lambda_\ell = \rk \mathbf{L}_\ell$.
As $\ker(\bG _\ell \rightarrow \mathbf{L}_\ell)^\circ =\mathbf{D}_\ell^\circ$, by part (i) of Theorem~\ref{l-adic},
$$\rk_\ell \Gamma_\ell = \rk_\ell \Lambda_\ell - \rk_\ell \Delta_\ell = \rk \mathbf{L}_\ell - \rk_\ell \Delta^\circ_\ell 
\ge \rk \mathbf{L}_\ell - \rk \mathbf{D}^\circ_\ell = \rk \bG _\ell.$$
We obtain
$$\rk_\ell\Gamma_\ell^{\sc}=\rk_\ell \Gamma_\ell \ge \rk \bG _\ell = \rk \bG _\ell^{\sc}.$$

By Theorem~\ref{l-adic}, $\Gamma_\ell^{\sc}$ is a hyperspecial maximal compact in $\bG _\ell^{\sc}(\Q_\ell)$. Since a reductive group $\bG $ over an $\ell$-adic field $F$ is unramified if and only if $\bG (F)$ contains a hyperspecial maximal compact subgroup \cite[$\mathsection1$]{Milne}, $\bG _\ell^{\sc}$ is unramified over $\Q_\ell$.
\end{proof}

\section{Some Applications and Extensions}

\begin{cor}
\label{high-rank}
If, for some $\ell_0$, the group obtained from $\bG _{\ell_0}^{\sc}$ by extension of scalars to 
$\bar\Q_{\ell_0}$ is a product of groups of the form $\SL_{m_i,\bar\Q_{\ell_0}}$,
where each $m_i$ satisfies $m_i=5,7$ or $m_i\ge 10$ and at most one $m_i=5$, then $\Gamma^{\sc}_\ell$ is
a hyperspecial maximal compact subgroup of $\bG _\ell^{\sc}(\Q_\ell)$ for all $\ell$ sufficiently large.
\end{cor}

\begin{proof}

By \cite{Hui1}, if $\bG _{\ell_0}$ is of type A, with the rank of every factor
belonging to the set $\{4,6\}\cup [9,\infty)$ and at most one factor of rank $4$, then $\bG _\ell$ is  of type A for all $\ell$.
\end{proof}

\begin{cor}
\label{global-A}
Let $\bG /\Q$ be a type A subgroup of $\GL_n$.  Suppose
$\Gamma_\ell \subset \bG (\Q_\ell)$ for all $\ell$, and $\Gamma_{\ell_0}$ is Zariski-dense in $\bG _{\Q_{\ell_0}}=\bG \times_\Q \Q_{\ell_0}$ 
for some $\ell_0$.  Then for all
$\ell$ sufficiently large, $\Gamma_\ell^{\sc}$ is a hyperspecial maximal subgroup of $\bG ^{\sc}(\Q_\ell)$.
\end{cor}

\begin{proof}
By \cite{Hui1}, $\rk \bG _\ell$ is independent of $\ell$ and therefore
always equal to 
$$\rk \bG _{\ell_0} = \rk \bG _{\Q_{\ell_0}} = \rk \bG .$$
As $\bG _\ell\subset \bG _{\Q_\ell}$,
Lemma \ref{A-lemma} implies $\bG _\ell^{\ss} = \bG _{\Q_\ell}^{\ss}$, so $\bG _\ell$ is of type A for all $\ell$.
The corollary now follows from Theorem~\ref{main}.
\end{proof}

%For example, $\bG $ could be of the form $\GL_k(D)$ for some division algebra $D$ over $\Q$
%or of the form $U(D,h)$, where $D$ is a simple algebra with an involution $\sigma$
%of the second kind and $h$ is a hermitian form
%relative to $\sigma$.

A variant of this result is the following:

\begin{cor}
\label{concrete}
Suppose that in Corollary~\ref{global-A} we know in addition that $\bG $ is connected reductive with
simply connected derived group $\mathbf{D}$ and $\Gamma_\ell\subset \bG (\Q_\ell)\cap \GL_n(\Z_\ell)$ for all $\ell$.  Then
for all $\ell\gg1$, we have
$$ \Gamma_\ell \cap \mathbf{D}(\Q_\ell) = \GL_n(\Z_\ell)\cap \mathbf{D}(\Q_\ell).$$
\end{cor}

\begin{proof}
Clearly
$$[\Gamma_\ell,\Gamma_\ell]\subset \Gamma_\ell \cap \mathbf{D}(\Q_\ell) \subset \GL_n(\Z_\ell)\cap \mathbf{D}(\Q_\ell),$$
so it suffices to prove that $[\Gamma_\ell,\Gamma_\ell]$ is a maximal compact subgroup of $\mathbf{D}(\Q_\ell)$.
As $\mathbf{D}$ is simply connected, the same is true for $\mathbf{D}_{\Q_\ell}$, and it follows that $\bG _{\Q_\ell}^{\sc} = \mathbf{D}_{\Q_\ell}$.
By Corollary~\ref{global-A}, $\bG _\ell^{\sc} = \mathbf{D}_{\Q_\ell}$.
If $x,y\in \Gamma_\ell\subset \bG (\Q_\ell)$ have the same image in $\bG ^{\ss}(\Q_\ell)$ as $x',y'\in \mathbf{D}(\Q_\ell)$ respectively,
then $[x,y] = [x',y']$.  It follows that $[\Gamma_\ell,\Gamma_\ell] = [\Gamma_\ell^{\sc},\Gamma_\ell^{\sc}]$.
As $\Gamma_\ell^{\sc}$ is a hyperspecial maximal compact for $\ell$ sufficiently large, it suffices to prove that
$\Gamma_\ell^{\sc}$ is perfect for $\ell\gg1$.

Let $\cD/\Z_\ell$ be a semisimple group scheme with $\cD_{\Q_\ell} = \mathbf{D}_{\Q_\ell}$ simply connected and $\Gamma_\ell^{\sc} = \cD(\Z_\ell)$.
Then $\cD_{\F_\ell}$ is simply connected, so for $\ell\ge 5$, $\cD(\F_\ell)$ is a perfect quasi-simple group.  Thus, $[\Gamma_\ell^{\sc},\Gamma_\ell^{\sc}]$
is a closed subgroup of $\cD(\Z_\ell)$ mapping onto $\cD(\F_\ell)$.  By \cite{Vasiu}, if $\ell\gg1$, $[\Gamma_\ell^{\sc},\Gamma_\ell^{\sc}]=\Gamma_\ell^{\sc}$.
\end{proof}

\begin{remark}
Suppose $\bG _\ell$ is connected and reductive. Let $\bG _\ell^{\der}$ be the derived group of $\bG _\ell$. We have a natural map $\pi_\ell:\bG _\ell^{\sc}\to \bG _\ell$. Then $\pi_\ell^{-1}(\Gamma_\ell)=\Gamma_\ell^{\sc}$ for $\ell\gg1$. Indeed, we have $\pi_\ell(\Gamma_\ell^{\sc})\subset \Gamma_\ell\cdot \mathbf{Z}_\ell(\Q_\ell)$ if $\mathbf{Z}_\ell$ denotes the center of $\bG _\ell$. Since $\Gamma_\ell^{\sc}$ is hyperspecial for $\ell\gg1$, it is perfect for $\ell\gg1$ from the proof of Corollary \ref{concrete}. Therefore, we obtain
$$\pi_\ell(\Gamma_\ell^{\sc})=[\pi_\ell(\Gamma_\ell^{\sc}),\pi_\ell(\Gamma_\ell^{\sc})]\subset [\Gamma_\ell,\Gamma_\ell]\subset \Gamma_\ell\cap\bG _\ell^{\der}(\Q_\ell).$$
This implies $\pi_\ell^{-1}(\Gamma_\ell)=\Gamma_\ell^{\sc}$ for $\ell\gg1$ since $\Gamma_\ell^{\sc}$ is maximal compact.
\end{remark}

Corollary~\ref{concrete} covers the classical $\GL_2$ case of Serre's original theorem.
It also covers the case that $\bG _\ell$ is unitary, which arises, for instance, from the cohomology of trielliptic curves.
Fix positive integers $d_1$ and $d_2$ in the same residue class (mod $3$), and consider the projective curve
$$X:y^3z^{d_1+2d_2-3} = \prod_{i=1}^{d_1} (x-u_i z)\prod_{j=1}^{d_2} (x-v_j z)^2 ,$$
defined over $K=\Q(\zeta_3)(u_1,\ldots,u_{d_1},v_1,\ldots,v_{d_2})$.  If $\Gamma_\ell$ denotes the image of $\Gal_K$ in $\Aut(H^1(\bar X,\Q_\ell))$, then
$\Gamma_\ell$ is contained in the centralizer in $\GSp_n(\Q_\ell)$ of an element of order $3$, the image of the automorphism
$(x:y:z)\mapsto (x:\zeta_3 y:z)$.  This centralizer is a unitary group $U_\ell$, and by a result of Achter and Pries, \cite{AP}, for every $\ell$, $\Gamma_\ell$ contains the derived group of $U_\ell$.
By $d_1+d_2$ iterations of the theorem of Cadoret and Tamagawa \cite{CT}, there exist pairwise distinct $a_i$ and $b_j$ in $\Q$ such that specializing $X$ to $u_i=a_i, v_j=b_j$, the resulting
curve over $\Q(\zeta_3)$ satisfies the condition that the Zariski closure of monodromy contains the derived group of $U_\ell$.  This curve is in fact obtained 
from a trielliptic curve $X_0$  over $\Q$ by extending scalars to $\Q(\zeta_3)$. Thus, Corollary~\ref{concrete} applies to $X_0$.
\medskip

Theorem~\ref{main}, Corollary~\ref{high-rank},  Corollary~\ref{global-A}, and Corollary~\ref{concrete} are also true in the following setting. Let $K$ be a finitely generated field extension of $\Q$ and $Y$ a smooth, separated and geometrically connected scheme over $K$. Let $f:X\rightarrow Y$ be a smooth, proper morphism with geometrically connected fibers. Fix an integer $i$. For any geometric point $\bar y$, the \'etale fundamental group $\pi_1(Y,\bar y)$ acts on the $i$th \'etale cohomology of the fiber $X_{\bar y}$ and we obtain a system of $\ell$-adic representations
\begin{equation*}
\rho_\ell \colon \pi_1(Y,\bar y)\rightarrow \Aut(H^i(X_{\bar y},\Q_\ell))\cong\GL_n(\Q_\ell).
\end{equation*}
Up to conjugation, the image $\Gamma_\ell:=\rho_\ell(\pi_1(Y,\bar y))$ and its Zariski closure $\bG _\ell$ in $\GL_n$ do not depend on the choice of $\bar y$. 
Note that in this setting, we assume the system comes from the full cohomology group of the fiber rather than a subsystem.

\begin{thm}
Theorem~\ref{main}, Corollary~\ref{high-rank},  Corollary~\ref{global-A}, and Corollary~\ref{concrete} are true for $\Gamma_\ell$ and $\bG _\ell$ in this setting.
\end{thm}

\begin{proof}
Let $\bar Y$ be $Y\times_K\bar K$ and $k(y)$ be the residue field of $y\in Y$. We have the commutative diagram
$$\xymatrix{
0 \ar[r] &\pi_1(\bar Y,\bar y) \ar[r] & \pi_1(Y,\bar y) \ar[r] & \mathrm{Gal}_K \ar[r]& 0\\
&&& \Gal_{k(y)} \ar@{_{(}->}[u] \ar@{_{(}->}[lu]^{\sigma(y)}}$$
with the first row exact, $\Gal_{k(y)}$ a closed subgroup of $\mathrm{Gal}_K$, and $\sigma(y)$  injective. The system of $\ell$-adic representations $\{\rho_\ell\circ \sigma(y)\}$ is isomorphic to the system of Galois actions of $\Gal_{k(y)}$ on $H^i(X_{\bar y},\Q_\ell)$. Denote the image of $\rho_\ell\circ\sigma(y)$ and its Zariski closure in $\GL_n$ by respectively $\Gamma_\ell^y$ and $\bG _\ell^y$.
We have inclusions $\Gamma_\ell^y\subset\Gamma_\ell$ and $\bG _\ell^y\subset \bG _\ell$.

Suppose $\bG _{\ell_0}$ is of type A for some prime $\ell_0$. 
Pick a closed point $y$ of $Y$ such that $\Gamma_{\ell_0}^y$ is open in $\Gamma_{\ell_0}$
(see \cite{letter,Ca}). 
Then the Zariski closure of $\rho_{\ell_0}(\pi_1(\bar Y, \bar y))$ in $\GL_n$ is also of type A. Then Corollary 1.4 of \cite{Ca} implies $\Gamma_\ell^y$ is open in $\Gamma_\ell$ for all $\ell$. Therefore, it suffices to prove the statements for $\{\rho_\ell\circ\sigma(y)\}$. Since $X_y\rightarrow \mathbf{Spec}(k(y))$ is the generic fiber of some smooth, proper morphism $X'\rightarrow Y'$  with geometrically connected fibers, where $Y'$ is smooth, separated, and geometrically connected over a number field $K'$, we repeat the above arguments to further the reduce system $\{\rho_\ell\circ\sigma(y)\}$ to a compatible system of Galois representations of \'etale cohomology arising from some complete, non-singular variety defined over $K'$, and in this setting we can appeal to Theorem~\ref{main}, Corollary~\ref{high-rank},  Corollary~\ref{global-A}, and Corollary~\ref{concrete}.
\end{proof}


\begin{thebibliography}{DDMS}

\bibitem[AP]{AP}
	Achter, Jeffrey D.; Pries, Rachel:
	The integral monodromy of hyperelliptic and trielliptic curves. 
	\textit{Math.\ Ann.} \textbf{338} (2007), no. 1, 187--206.
	
\bibitem[BT]{BT}
	Borel, A.; Tits, J.:
	\'El\'ements unipotents et sous-groupes paraboliques de groupes r\'eductifs. I.
	\textit{Invent.\ Math.} 
	\textbf{12} (1971), 95--104.

\bibitem[Bo]{Bo}
	Bourbaki, N.:
	\'El\'ements de math\'ematique. Fasc. XXXVII. 
	Groupes et alg\`ebres de Lie. 
	Chapitre II: Alg\`ebres de Lie libres. Chapitre III: Groupes de Lie. 
	Actualit\'es Scientifiques et Industrielles, No. 1349. Hermann, Paris, 1972.
		
\bibitem[Ca]{Ca}
	Cadoret, Anna:
	On $\ell$-independency in families of motivic $\ell$-adic representations,
  preprint.

\bibitem[CT]{CT} 
	Cadoret, A.; Tamagawa, A.: 
	A uniform open image theorem for $\ell$-adic representations I, 
	\emph{Duke Math.\ J.} \textbf{161}, p. 2605-2634, 2012.
    			
\bibitem[DG]{SGA3}
	Demazure, Michael; Grothendieck, Alexander: 
	Sch\'emas en groupes (SGA 3), 
	Lecture Notes in Mathematics 153, 
	Springer-Verlag, Berlin, 1970.
	
\bibitem[DP]{DP}
	Devic, Anna; Pink, Richard:
	Adelic openness for Drinfeld modules in special characteristic. 
	\textit{J.\ Number Theory}
	\textbf{132} (2012), no.\ 7, 1583--1625.

\bibitem[DDMS]{DDMS}
	Dixon, J. D.; du Sautoy, M. P. F.; Mann, A.; Segal, D.: 
	Analytic pro-p Groups, 
	London Mathematical Society Lecture Note Series 157, 
	Cambridge University Press, 1991.
	
%\bibitem[Gr]{EGA}
	%Grothendieck, Alexandre: 
	%\'El\'ements de g\'eom\'etrie alg\'ebrique. IV. \'Etude locale des sch\'emas et des morphismes de sch\'emas. 
	%III. \textit{Inst.\ Hautes \'Etudes Sci.\ Publ.\ Math.} \textbf{28} (1966), 5--255.	
	
\bibitem[Ha]{Hall}
	Hall, Chris:
	Big symplectic or orthogonal monodromy modulo l. 
	\textit{Duke Math.\ J.} 
	\textbf{141} (2008), no.\ 1, 179--203. 
	
\bibitem[Hu1]{Hui1}
	Hui, Chun Yin:
	Monodromy of Galois representations and equal-rank subalgebra equivalence,
	\textit{Math. Res. Lett.} \textbf{20} (2013), no. 4, 705-728.

\bibitem[Hu2]{Hui2}
	Hui, Chun Yin:
	$\ell$-independence for compatible systems of (mod $\ell$) Galois representations,
	arXiv:1305.2001, preprint.

\bibitem[La]{Larsen} 
	Larsen, Michael:
	Maximality of Galois actions for compatible systems. 
	\textit{Duke Math.\ J.} 
	\textbf{80} (1995), no.\ 3, 601--630. 

%\bibitem[La2]{Larsen2}
	%Larsen, Michael:
	%Exponential generation and largeness for compact $p$-adic Lie groups. 
	%\textit{Algebra Number Theory}
	%\textbf{4} (2010), no.\ 8, 1029--1038.

\bibitem[LP]{LP}
	Larsen, Michael; Pink, Richard: 
	On $\ell$-independence of algebraic monodromy groups in compatible systems of representations. 
	\textit{Invent.\ Math.} \textbf{107} (1992), no. 3, 603--636.

\bibitem[Mi]{Milne}
	Milne, James: 
	The points on a Shimura variety modulo a prime of good reduction. 
	The zeta functions of Picard modular surfaces, Publications du CRM, 1992, p. 151--253.
		
\bibitem[No]{Nori}
	 Nori, Madhav V.:
	 On subgroups of $\GL_n(\F_p)$. 
	 \textit{Invent.\ Math.}
	 \textbf{88} (1987), no.\ 2, 257--275.

\bibitem[Pi]{Pink}
	Pink, Richard:
	Compact subgroups of linear algebraic groups. 
	\textit{J.\ Algebra} \textbf{206} (1998), no.\ 2, 438--504.
	
\bibitem[PR]{PR}
	Pink, Richard; R\"utsche, Egon:
	Adelic openness for Drinfeld modules in generic characteristic. 
	\textit{J.\ Number Theory} \textbf{129} (2009), no.\ 4, 882--907.
	
\bibitem[Ri1]{Ribet1}
	Ribet, Kenneth A.:
	Galois action on division points of Abelian varieties with real multiplications. 
	\textit{Amer.\ J.\ Math.} 
	\textbf{98} (1976), no.\ 3, 751--804. 	 

\bibitem[Ri2]{Ribet2}
	Ribet, Kenneth A.:
	On l-adic representations attached to modular forms. II. 
	\textit{Glasgow Math.\ J.} 
	\textbf{27} (1985), 185--194.	

\bibitem[Se1]{Serre64}
	Serre, Jean-Pierre:
	Groupes analytiques p-adiques. 
	1964 S\'eminaire Bourbaki, 1963/64, Fasc. 2, Expos\'e 270.
	
%\bibitem[Se2]{SerreLA}
%	 Serre, Jean-Pierre:
%	 Groupes de Lie l-adiques attach\'es aux courbes elliptiques. 
%	 Colloque CNRS
%	 \textbf{143} (1966), 239--256.
%	 
%\bibitem[Se3]{Serre66}
%	Serre, Jean-Pierre:
%	R\'esum\'e des cours de 1965--1966.
%	Annuaire du Coll\`ege de France (1966), 49--58.

\bibitem[Se2]{ALR}
	Serre, Jean-Pierre:
	Abelian l-adic representations and elliptic curves. 
	McGill University lecture notes written with the collaboration of Willem Kuyk and John Labute.
	W. A. Benjamin, Inc., New York-Amsterdam, 1968.
	
\bibitem[Se3]{Se}
	Serre, Jean-Pierre:
	Propri\'et\'es galoisiennes des points d'ordre fini des courbes elliptiques. 
	\textit{Invent.\ Math.}
	\textbf{15} (1972), no.\ 4, 259--331.

\bibitem[Se4]{LF}
	Serre, Jean-Pierre:
	Local fields.  Translated from the French by Marvin Jay Greenberg.
	Graduate Texts in Mathematics, 67. Springer-Verlag, New York-Berlin, 1979.
	
\bibitem[Se5]{letter}
	Serre, Jean-Pierre:
	Letter to K. A. Ribet, Jan. 1, 1981, 
	reproduced in Coll. Papers, vol. IV, no 133.

\bibitem[Se6]{Serre86}
	Serre, Jean-Pierre:
	R\'esum\'e des cours de 1985--1986.
	Annuaire du Coll\`ege de France (1986), 95--99.
	
\bibitem[Se7]{GC}
	Serre, Jean-Pierre:
	Galois cohomology. Translated from the French by Patrick Ion and revised by the author. 
	Springer-Verlag, Berlin, 1997. 

\bibitem[Sp]{Springer}
	Springer, T. A.:
		Reductive groups. Automorphic forms, representations and L-functions (Proc. Sympos. Pure Math., Oregon State Univ., Corvallis, Ore., 1977), Part 1, pp. 3--27, Proc. Sympos. Pure Math., XXXIII, Amer. Math. Soc., Providence, R.I., 1979.	
\bibitem[St]{Steinberg}
	 Steinberg, Robert:
	 Endomorphisms of linear algebraic groups. Memoirs of the American Mathematical Society, No. 80 American Mathematical Society, Providence, R.I., 1968.

\bibitem[Ti]{Tits}
	Tits, Jacques: 
	Reductive groups over local fields. Automorphic forms, representations and L-functions (Proc. Sympos. Pure Math., Oregon State Univ., Corvallis, Ore., 1977), Part 1, pp. 29--69, 
	Proc. Sympos. Pure Math., XXXIII, Amer. Math. Soc., Providence, R.I., 1979.
	
\bibitem[Va]{Vasiu}
	Vasiu, Adrian: 
	Surjectivity criteria for p-adic representations. I. 
	\textit{Manuscripta Math.} \textbf{112} (2003), no.\ 3, 325--355. 
	
\end{thebibliography}
\end{document}